\documentclass[12pt,reqno]{amsart}
\usepackage{amssymb}
\usepackage{times}
\newtheorem{thm}{Theorem}[section]
\newtheorem{lemma}[thm]{Lemma}
\newtheorem{cor}[thm]{Corollary}
\newtheorem{prop}[thm]{Proposition}

\theoremstyle{definition}

\newtheorem{eg}[thm]{Example}

\theoremstyle{remark}
\newtheorem*{rmk*}{Remark}
\newtheorem{rmk}[thm]{Remark}

\def\dim{\operatorname{dim}\nolimits}

\def\Hom{\operatorname{Hom}\nolimits}

\def\Ext{\operatorname{Ext}\nolimits}

\newcommand{\la}{\lambda}
\newcommand{\ot}{\otimes}
\DeclareMathOperator{\car}{char}
\DeclareMathOperator{\rank}{rank}

\newcommand{\Gt}{{\widetilde{G}}}
\newcommand{\Tt}{\widetilde{T}}

\DeclareMathOperator{\SO}{SO}
\DeclareMathOperator{\Spin}{Spin}
\DeclareMathOperator{\SL}{SL}
\DeclareMathOperator{\Sp}{Sp}
\DeclareMathOperator{\PGL}{PGL}
\DeclareMathOperator{\HSpin}{HSpin}

\DeclareMathOperator{\Lie}{Lie}

\newcommand{\C}{\mathbb{C}}
\newcommand{\ZZ}{\mathbb{Z}}
\newcommand{\QQ}{\mathbb{Q}}
\newcommand{\NN}{\mathbb{N}}

\newcommand{\FF}{\mathbb{F}}

\newcommand{\hvee}{h^\vee}

\newcommand{\hsr}{\alpha_0}     
\newcommand{\hr}{\widetilde{\alpha}} 

\newcommand{\ft}{\mathfrak{t}}

\newcommand{\fn}{\mathfrak{n}}

\newcommand{\df}{\mathrm{d}f}

\usepackage{color}

\usepackage{hyperref}
\hypersetup{
  colorlinks=true, 
   linkcolor=red,  
}
\usepackage[hyphenbreaks]{breakurl}

\begin{document}

\title{Globally Irreducible Weyl modules}

\author[S. Garibaldi]{Skip Garibaldi}
\address{Center for Communications Research, San Diego, California 92121}
\email{skip@member.ams.org}

\author[R.M. Guralnick]{Robert M. Guralnick}
\address{Department of Mathematics, University of Southern California,
Los Angeles, CA 90089-2532}
\email{guralnic@usc.edu}

\author[D.K. Nakano]{Daniel K. Nakano}
\address{Department of Mathematics, University of Georgia, Athens, Georgia 30602, USA}
\email{nakano@math.uga.edu}

\dedicatory{Dedicated to Benedict Gross}

\thanks{The second author was partially supported by NSF grants DMS-1265297 and DMS-1302886.  Research of the third author was partially supported by NSF grants DMS-1402271 and DMS-1701768.}


\begin{abstract}
In the representation theory of split reductive algebraic groups, it is well known that every Weyl module with minuscule highest weight is irreducible over every field. 
Also, the adjoint representation of $E_8$ is irreducible over every field.  In this paper, we prove a converse to these statements, as conjectured by Gross: if a Weyl module is irreducible over every field, it must be either one of these, or trivially constructed from one of these.  We also prove a related result on non-degeneracy of the reduced Killing form.
\end{abstract}

\maketitle

\section{Introduction}

Split semisimple linear algebraic groups over arbitrary fields can be viewed as a generalization of semisimple Lie algebras over the complex numbers, or even compact real Lie groups.  As with Lie algebras, such algebraic groups are classified up to isogeny by their root system.  Moreover, the set of  irreducible representations of such a group is in bijection with the cone of dominant weights for the root system and the representation ring --- i.e., $K_0$ of the category of finite-dimensional representations --- is a polynomial ring with generators corresponding to a basis of the cone.

One way in which this analogy breaks down is that, for an algebraic group $G$ over a field $k$ of \emph{prime} characteristic, in addition to the irreducible representation $L(\la)$ corresponding to a dominant weight $\la$, there are three other representations naturally associated with $\lambda$, namely the standard module $H^0(\la)$, the Weyl module $V(\la)$, and the tilting module $T(\la)$.\footnote{The definitions of these three modules make sense also when $\car k = 0$, and in that case all four modules are isomorphic.} The definition of $H^0(\la)$ is particularly simple: view $k$ as a one-dimensional representation of a Borel subgroup $B$ of $G$ where $B$ acts via the character $\la$, then define $H^0(\la):=\text{ind}_{B}^{G}\lambda$ 
to be the induced $G$-module.  The \emph{Weyl module} $V(\la)$ is the dual of $H^0(-w_0\la)$ for $w_0$ the longest element of the Weyl group and has head $L(\la)$. Typical examples of Weyl modules are $\Lie(G)$ for $G$ semisimple simply connected ($V(\la)$ for $\la$ the highest root) and the natural module of $\SO_n$.  See \cite{Jan} for general background on these three families of representations.

It turns out that if any two of the four representations $L(\la)$, $H^0(\la)$, $V(\la)$, $T(\la)$ are isomorphic over a given field $k$, then all four are.  Our focus is on the question: for which $\la$ are all four isomorphic for \emph{every} field $k$? 

This can be interpreted as a question about representations of split reductive group schemes over $\ZZ$.  Recall that isomorphism classes of such groups are in bijection with (reduced) root data as described in \cite[XXIII.5.2]{SGA3:new}.  A root datum for a group $G$ includes a character lattice $X(T)$ of a split maximal torus $T$ and the set $R \subset X(T)$ of roots of $G$ with respect to $T$.  Picking an ordering on $R$ specifies a cone of dominant weights $X(T)_+$ in $X(T)$.  For each $\la \in X(T)_+$, there is a representation $V(\la)$ for $G$, defined over $\ZZ$, that is generated by a highest weight vector with weight $\la$ such that $V(\la) \ot \C$ is the irreducible representation with highest weight $\la$ of the complex reductive group $G \times \C$ and for every field $k$, $V(\la) \ot k$ is the Weyl module of $G \times k$ mentioned above, see \cite[II.8.3]{Jan} or \cite[p.~212]{St}.  Consequently, the question in the preceding paragraph is the same as asking: \emph{For which $G$ and $\la$ is it true that $V(\la) \ot k$ is an irreducible representation of $G \times k$ for every field $k$?}  Because $G$ is split, $V(\la) \ot k$ is irreducible if and only if $V(\la) \ot P$ is irreducible where $P$ is the prime field of $k$\footnote{See \cite[II.2.9]{Jan}.  For a detailed study of how this fails when $G$ is not split, see \cite{Ti:R}.}, it is natural to call such $V(\la)$ \emph{globally irreducible}.

%
There is a well known and elementary sufficient criterion: 
\begin{equation} \label{sufficient}
\text{\emph{If $\la$ is minuscule, then $V(\la) \ot k$ is irreducible for every field $k$.}}
\end{equation}
See \S\ref{defs} for the definition of minuscule.  This provides an important family of examples, because representations occurring in this way include $\Lambda^r(V)$ for $1 \le r < n$ where $V$ is the natural module for $\SL_n$; the natural modules for $\SO_{2n}$, $\Sp_{2n}$, $E_6$ and $E_7$; and the (half) spin representations of $\Spin_n$.

While these representations play an outsized role, it is nevertheless true that in any reasonable sense they are a set of measure zero among the list of irreducible representations.  Therefore, we were surprised when Benedict Gross 
proposed to us that the sufficient condition \eqref{sufficient} is quite close to also being a \emph{necessary} condition, i.e., that there is only one other example.  The purpose of this paper is to prove his claim, which is the following theorem.

%
%

\begin{thm}  \label{MT}
Let $G$ be a split, simple algebraic group over $\ZZ$ with split maximal torus $T$ and fix $\la \in X(T)_+$. In the following cases, $V(\la) \ot k$ is irreducible for every field $k$:
	\begin{enumerate}
	\renewcommand{\theenumi}{\alph{enumi}}
	\item  \label{MT.min} $\lambda$ is a \underline{minuscule} dominant weight, or 
	\item \label{MT.E8} $G$ is a group of type $E_8$ and $\lambda$ is the highest root (i.e., $V(\lambda)$ is the adjoint representation for $E_{8}$); 
	\end{enumerate}
Otherwise, there is a prime $p \le 2(\rank G) + 1$ such that $V(\la) \ot k$ is a reducible representation of $G$ for every field $k$ of characteristic $p$.
\end{thm} 


The bound $2(\rank G) + 1$ is sharp by Theorem \ref{B.thm} below. The case where $G$ is simple and simply connected (as in Theorem \ref{MT}) is the main case. 
We have stated the theorem with these simplified hypotheses for the sake of clarity.  See \S\ref{defs} for a discussion of the more general version where $G$ is assumed merely to be reductive.

One surprising feature of our proof is the method we use to address a particular Weyl module of type $B$ in \S\ref{B.sec}, which we settle by appealing to modular representation theory of finite groups.

The literature contains some results complementary to Theorem \ref{MT}, although we do not use them in our proof.  For $G$ of type $A$, Jantzen gave in \cite[II.8.21]{Jan} a necessary and sufficient condition for the Weyl module $V(\la)$ to be irreducible over fields of characteristic $p$.  McNinch \cite{McNinch:ss} (extending Jantzen \cite{Jantzen:low}) showed that for simple $G$ and for $\dim V(\la) \le (\car k) \cdot (\rank G)$, $V(\la)$ is irreducible.

We remark that John Thompson asked in \cite{Th} an analogous question where $G$ is finite: for which $\ZZ[G]$-lattices $L$ is $L/pL$ irreducible for every prime $p$?  This was extended by Gross to the notion of globally irreducible
representations, see \cite{Gross:JAMS} and \cite{Tiep}. Our results demonstrate that $F_{4}$ and $G_{2}$ are the only groups that do not admit globally irreducible representations
other than the trivial representation. 

In an appendix, we prove another result that is similar in flavor to Theorem \ref{MT}: we determine the split simple $G$ over $\ZZ$ such that the reduced Killing form on $\Lie(G) \ot k$ is nondegenerate for every field $k$.  This is done by calculating the determinant of the form on $\Lie(G)$, completing the calculation for $G$ simply connected in \cite{SpSt}.

\subsection*{Quasi-minuscule representations}
The representations appearing in \eqref{MT.min} and \eqref{MT.E8} of Theorem \ref{MT} are \emph{quasi-minuscule} (called ``basic'' in \cite{Matsumoto:basic}), meaning that the non-zero weights are a single orbit under the Weyl group.  For $G$ simple, the quasi-minuscule Weyl modules are the $V(\la)$ with $\la$ minuscule or equal to the highest short root $\hsr$.

It is not hard to see that $V(\hsr) \ot k$ is reducible for some $k$ when $G$ is not of type $E_8$.  If $G$ has type $A$, $D$, $E_6$, or $E_7$, then $V(\hsr)$ is the action of $G$ on the Lie algebra of its simply connected cover $\Gt$, and the Lie algebra of the center $Z$ of $\Gt$ is a nonzero invariant submodule when $\car k$ divides the exponent of $Z$.  The case where $G$ has type $B$ or $C$ is discussed in \S\ref{fund.sec}.  If $G$ has type $G_2$ or $F_4$, then $V(\hsr)$ is the space of trace zero elements in an octonion or Albert algebra, and the identity element generates an invariant subspace if $\car k = 2$ or 3 respectively.

\subsection*{Acknowledgements} The authors thank Dick Gross for suggesting the problem that led to the formulation of Theorem~\ref{MT}\eqref{MT.min}\eqref{MT.E8}, and for several useful discussions pertaining to the contents of the paper. 
The authors also thank Henning Andersen, James Humphreys, Jens Carsten Jantzen, George Lusztig, and the referee for their suggestions and comments on an earlier version of this manuscript.
\section{Definitions and notation}  \label{defs}
We will follow the notation and conventions presented in \cite{Jan}.  When we refer to an algebraic group $G$, we mean a smooth affine group scheme of finite type as in \cite{SGA3:new} or \cite[Ch.~VI]{KMRT}, as opposed to its (abstract) group of $k$-points, which we denote by $G(k)$.  An example of this difference is that the natural map $\SL_p \to \PGL_p$ has nontrivial kernel the group scheme $\mu_p$, yet for $k$ a field of prime characteristic $p$, the map $\SL_p(k) \to \PGL_p(k)$ is injective.

Let $G$ be a simple simply connected 
algebraic group, $T$ be a maximal split torus of $G$ and $\Phi$ be the root system 
associated to $(G,T)$.  Fix a choice of simple roots $\Delta$.
Let $B$ be a Borel subgroup containing $T$ corresponding 
to the negative roots and let $U$ denote the unipotent radical of $B$. 

One can naturally view $\Phi$ as contained in a Euclidean space ${\mathbb E}$ with
 inner product $\langle\ , \ \rangle$. Let $X(T)$ be the integral weight lattice 
obtained from $\Phi$. The set $X(T)$ has a partial
ordering defined as follows. If $\lambda,\mu\in X(T)$, then
$\lambda\geq \mu$ if and only if $\lambda - \mu\in \sum_{\alpha\in
\Delta}\mathbb{N}\alpha$.

\begin{table}[bth]
{\centering\noindent\makebox[450pt]{
\begin{tabular}[c]{p{2.2in}|p{2.2in}}

$\small{(A_n)~~}$
\begin{picture}(7,2)(0,0)
\multiput(0,1)(20,0){3}{\circle{6}}
\multiput(62,1)(20,0){3}{\circle{6}}
\put(0,1){\circle*{3}}
\put(0,1){\line(1,0){20}}
\put(20,1){\circle*{3}}
\put(20,1){\line(1,0){20}}
\put(40,1){\circle*{3}}
\put(40,-1.6){ \mbox{$\cdots$}}
\put(62,1){\circle*{3}}
\put(62,1){\line(1,0){20}}
\put(82,1){\circle*{3}}
\put(82,1){\line(1,0){20}}
\put(102,1){\circle*{3}}

\put(-2,-7){\mbox{\tiny $1$}}
\put(18,-7){\mbox{\tiny $2$}}
\put(38,-7){\mbox{\tiny $3$}}
\put(54,-7){\mbox{\tiny $n$$-$$2$}}
\put(75,-7){\mbox{\tiny $n$$-$$1$}}
\put(100,-7){\mbox{\tiny $n$}}
\end{picture}
\vspace{0.5cm}

&

$\small{(E_6)~~}$
\begin{picture}(7,2)(0,0)
\put(0,-5){\circle*{3}}
\put(0,-5){\circle{6}}
\put(0,-5){\line(1,0){15}}
\put(15,-5){\circle*{3}}
\put(15,-5){\line(1,0){15}}
\put(30,-5){\circle*{3}}
\put(30,10){\circle*{3}}
\put(30,-5){\line(0,1){15}}
\put(30,-5){\line(1,0){15}}
\put(45,-5){\circle*{3}}
\put(45,-5){\line(1,0){15}}
\put(60,-5){\circle*{3}}
\put(60,-5){\circle{6}}

\put(-2,-13){\mbox{\tiny $1$}}
\put(13,-13){\mbox{\tiny $3$}}
\put(28,-13){\mbox{\tiny $4$}}
\put(43,-13){\mbox{\tiny $5$}}
\put(58.5,-13){\mbox{\tiny $6$}}
\put(33,9){\mbox{\tiny $2$}}
\put(22,9){\mbox{\tiny $\star$}}

\end{picture}

\\

$\small{(B_n)~~}$
\begin{picture}(7,2)(0,0)
\put(0,1){\circle*{3}}
\put(0,1){\line(1,0){20}}
\put(20,1){\circle*{3}}
\put(20,1){\line(1,0){20}}
\put(40,1){\circle*{3}}
\put(40,-1.6){ \mbox{$\cdots$}}
\put(62,1){\circle*{3}}
\put(62,1){\line(1,0){20}}
\put(82,1){\circle*{3}}
\put(82,2){\line(1,0){20}}
\put(82,0){\line(1,0){20}}
\put(89,-1){{\tiny\mbox{$>$}}}
\put(102,1){\circle*{3}}
\put(102,1){\circle{6}}

\put(-2,-7){\mbox{\tiny $1$}}
\put(-2,5){\mbox{\tiny $\star$}}
\put(18,-7){\mbox{\tiny $2$}}
\put(38,-7){\mbox{\tiny $3$}}
\put(54,-7){\mbox{\tiny $n$$-$$2$}}
\put(75,-7){\mbox{\tiny $n$$-$$1$}}
\put(100,-7){\mbox{\tiny $n$}}
\end{picture}
\vspace{0.5cm}

&

$\small{(E_7)~~}$
\begin{picture}(7,2)(0,0)
\put(0,-5){\circle*{3}}
\put(0,-5){\circle{6}}
\put(0,-5){\line(1,0){15}}
\put(15,-5){\circle*{3}}
\put(15,-5){\line(1,0){15}}
\put(30,-5){\circle*{3}}
\put(30,-5){\line(1,0){15}}
\put(45,-5){\circle*{3}}
\put(45,-5){\line(1,0){15}}
\put(45,10){\circle*{3}}
\put(45,-5){\line(0,1){15}}
\put(60,-5){\circle*{3}}
\put(60,-5){\line(1,0){15}}
\put(75,-5){\circle*{3}}

\put(-2,-13){\mbox{\tiny $7$}}
\put(13,-13){\mbox{\tiny $6$}}
\put(28,-13){\mbox{\tiny $5$}}
\put(43,-13){\mbox{\tiny $4$}}
\put(58.5,-13){\mbox{\tiny $3$}}
\put(73,-13){\mbox{\tiny $1$}}
\put(73,0){\mbox{\tiny $\star$}}

\put(47.5,9){\mbox{\tiny $2$}}
\end{picture}

\\

$\small{(C_n)~~}$
\begin{picture}(7,2)(0,0)
\put(0,1){\circle*{3}}
\put(0,1){\line(1,0){20}}
\put(20,1){\circle*{3}}
\put(20,1){\line(1,0){20}}
\put(40,1){\circle*{3}}
\put(40,-1.6){ \mbox{$\cdots$}}
\put(62,1){\circle*{3}}
\put(62,1){\line(1,0){20}}
\put(82,1){\circle*{3}}
\put(82,2){\line(1,0){20}}
\put(82,0){\line(1,0){20}}
\put(89,-1){{\tiny\mbox{$<$}}}
\put(102,1){\circle*{3}}
\put(0,1){\circle{6}}

\put(-2,-7){\mbox{\tiny $1$}}
\put(18,-7){\mbox{\tiny $2$}}
\put(18,5){\mbox{\tiny $\star$}}
\put(38,-7){\mbox{\tiny $3$}}
\put(54,-7){\mbox{\tiny $n$$-$$2$}}
\put(75,-7){\mbox{\tiny $n$$-$$1$}}
\put(100,-7){\mbox{\tiny $n$}}
\end{picture}
\vspace{0.5cm}

&

$\small{(E_8)~~}$
\begin{picture}(7,2)(0,0)
\put(0,-5){\circle*{3}}
\put(0,-5){\line(1,0){15}}
\put(15,-5){\circle*{3}}
\put(15,-5){\line(1,0){15}}
\put(30,-5){\circle*{3}}
\put(30,-5){\line(1,0){15}}
\put(45,-5){\circle*{3}}
\put(60,-5){\line(0,1){15}}
\put(60,-5){\circle*{3}}
\put(60,10){\circle*{3}}
\put(75,-5){\circle*{3}}
\put(75,-5){\line(1,0){15}}
\put(90,-5){\circle*{3}}
\put(45,-5){\line(1,0){15}}
\put(60,-5){\line(1,0){15}}

\put(-2,-13){\mbox{\tiny $8$}}
\put(-2,0){\mbox{\tiny $\star$}}

\put(13,-13){\mbox{\tiny $7$}}
\put(28,-13){\mbox{\tiny $6$}}
\put(43,-13){\mbox{\tiny $5$}}
\put(58.5,-13){\mbox{\tiny $4$}}
\put(73,-13){\mbox{\tiny $3$}}
\put(88,-13){\mbox{\tiny $1$}}
\put(62.5,9){\mbox{\tiny $2$}}
\end{picture}

\\

$\small{(D_n)~~}$
\begin{picture}(7,2)(0,0)
\put(0,1){\circle*{3}}
\put(0,1){\circle{6}}
\put(0,1){\line(1,0){20}}
\put(20,1){\circle*{3}}
\put(20,1){\line(1,0){20}}
\put(40,1){\circle*{3}}
\put(40,-1.6){ \mbox{$\cdots$}}
\put(62,1){\circle*{3}}
\put(62,1){\line(1,0){20}}
\put(82,1){\circle*{3}}
\put(82,2){\line(4,3){15}}
\put(82,0){\line(4,-3){15}}
\put(96.5,12.9){\circle*{3}}
\put(96.5,12.9){\circle{6}}
\put(96.5,-10.9){\circle*{3}}
\put(96.5,-10.9){\circle{6}}

\put(-2,-7){\mbox{\tiny $1$}}

\put(18,5){\mbox{\tiny $\star$}}
\put(18,-7){\mbox{\tiny $2$}}
\put(38,-7){\mbox{\tiny $3$}}
\put(54,-7){\mbox{\tiny $n$$-$$3$}}
\put(86,-0.5){\mbox{\tiny $n$$-$$2$}}
\put(100,-12){\mbox{\tiny $n$}}
\put(100,11.8){\mbox{\tiny $n$$-$$1$}}
\end{picture}

&

$\small{(F_4)~~}$
\begin{picture}(7,2)(0,0)
\put(0,1){\circle*{3}}
\put(0,1){\line(1,0){15}}
\put(15,1){\circle*{3}}
\put(15,0){\line(1,0){15}}
\put(15,2){\line(1,0){15}}
\put(19,-1){{\tiny\mbox{$>$}}}
\put(30,1){\circle*{3}}
\put(30,1){\line(1,0){15}}
\put(45,1){\circle*{3}}

\put(-2,-7){\mbox{\tiny $1$}}
\put(13,-7){\mbox{\tiny $2$}}
\put(28,-7){\mbox{\tiny $3$}}
\put(43,-7){\mbox{\tiny $4$}}
\put(43,5){\mbox{\tiny $\star$}}

\put(65,1){$\small{(G_2)~~}$}
\put(92,1){\circle*{3}}
\put(92,0.1){\line(1,0){15}}
\put(92,1.1){\line(1,0){15}}
\put(92,2.1){\line(1,0){15}}
\put(96,-1){{\tiny\mbox{$<$}}}
\put(107,1){\circle*{3}}

\put(90,-7){\mbox{\tiny $1$}}
\put(90,5){\mbox{\tiny $\star$}}

\put(105,-7){\mbox{\tiny $2$}}
\end{picture}

\end{tabular}
}}
\vskip .5cm
\caption{Dynkin diagrams of simple root systems, with simple roots numbered.  A circle around vertex $i$ indicates that the fundamental weight $\omega_i$ is minuscule.  A $\star$ indicates that $\omega_i$ is the highest short root $\hsr$.  The highest short root of $A_n$ is $\omega_1 + \omega_n$.} \label{dynks.table}
\end{table}

For $\alpha^{\vee}:=\frac{2\alpha}{\langle\alpha,\alpha\rangle}$ the coroot
corresponding to $\alpha\in \Phi$, 
the set of dominant integral weights is defined by
$$X(T)_{+}:=\{\lambda\in X(T):\ 0\leq \langle\lambda,\alpha^{\vee}\rangle\
\text{for all $\alpha \in \Delta$} \}.$$
The fundamental weights $\omega_{j}$ for $j=1,2,\dots, n$ are the dual basis to 
the  simple coroots. That is, if $\Delta=\{\alpha_{1},\alpha_{2},\dots,\alpha_{n}\}$ then 
$\langle \omega_{i},\alpha_{j}^{\vee} \rangle =\delta_{i,j}$.  

We call the weights in $X(T)_+$ that are minimal with respect to the partial ordering \emph{minuscule} weights.  Note that the zero weight is minuscule by this definition (in some references this is not the case).  Every nonzero minuscule weight is a fundamental dominant weight (one of the $\omega_i$'s), and we have marked them in Table \ref{dynks.table}.  We remark that there is a unique minuscule weight in each coset of the root lattice $\ZZ\Phi$ in the weight lattice $X(T)$ by \cite[\S{VI.2}, Exercise 5a]{Bou:g4} or \cite[\S13, Exercise 13]{Hum:Lie}; this can be an aid for remembering the number of minuscule weights for each type and for determining which minuscule weight lies below a given dominant weight.

\subsection*{Generalization of Theorem \ref{MT} to split reductive groups} 

Suppose now that $G$ is a split reductive group over a field $k$.  Then there is a unique split reductive group scheme over $\ZZ$ whose base change to $k$ is $G$, which we denote also by $G$; it is the split reductive group scheme over $\ZZ$ with the same root datum as $G$.  Moreover, there is a split reductive group scheme $G'$ over $\ZZ$ with a central isogeny $G' \to G$ where $G' = \prod_{i=0}^r G_i$ for $G_0$ a torus and $G_i$ simple and simply connected for $i \ne 0$, cf.~\cite[XXI.6.5.10]{SGA3:new}.  A Weyl module $V(\la)$ for $G$ restricts to a Weyl module $V(\sum \la_i)$ for $G'$, where $\la_i$ denotes the restriction of $\la$ to a maximal torus in $G_i$, and as in \cite[Lemma I.3.8]{Jan} we have $V(\sum \la_i) \cong \ot_{i=0}^r V(\la_i)$ where $V(\la_0)$ is one-dimensional.  Therefore, $V(\la) \ot k$ is an irreducible $G$-module for every field $k$ if and only if $V(\la_i) \ot k$ is an irreducible $G_i$-module for every $k$, i.e., if and only if $(G_i, V(\la_i))$ satisfies condition \eqref{MT.min} or \eqref{MT.E8} of Theorem \ref{MT} for all $i \ne 0$.


\section{Restriction to Levi subgroups}
For $J\subseteq \Delta$, let $L_{J}$ be the Levi subgroup of $G$ generated by the maximal torus $T$ and the root subgroups corresponding to roots that are linear combinations of elements of $J$.  Set 
\[
X_{J}(T)_{+}:=\{\lambda\in X(T)\ :\ 0\leq \langle \lambda,\alpha^{\vee} \rangle \ \text{for all $\alpha\in J$}\}.\]
For $\lambda\in X_{J}(T)_{+}$, we can construct an induced module $H^{0}_{J}(\lambda):=\text{ind}^{L_{J}}_{L_{J}\cap B}\lambda$ 
with simple $L_{J}$-socle $L_{J}(\lambda)$, and dually a Weyl module $V_{J}(\lambda)$ with head 
$L_{J}(\lambda)$.

\begin{thm}\label{thm:Levireduction} 
Let $G$ be a simple simply connected algebraic group and $J\subseteq \Delta$.
If $V(\lambda) \ot k$ is an irreducible $G$-module, then $V_J(\la) \ot k$ is an irreducible $L_J$-module.
\end{thm}

\begin{proof} 
For $k$ of characteristic 0,  $V_J(\lambda)$ is just the set of fixed points of $Q_J$ on $V(\lambda)$  (the unipotent
radical of the parabolic $P_J=L_JQ_J$);  this is part of \cite{Smith}.
Taking a $\ZZ$-form and reducing modulo $p$, we see that the dimension of the space of fixed points of $Q_J$ on $V(\lambda)$ can only go up in characteristic $p$.

So if $V(\lambda)=L(\lambda)$,  then again by \cite{Smith},  the fixed points of $Q_J$
on this module is $L_J(\lambda)$ but has dimension at least $V_J(\lambda)$. The 
other inequality is clear since $L_J(\lambda)$ is a quotient
of $V_J(\lambda)$, so $L_J(\lambda)=V_J(\lambda)$.
\end{proof}

\begin{rmk}
Given a group $G$ and a particular prime $p$, there are few known necessary and sufficient conditions in terms of $\lambda$ for the Weyl module $V(\lambda) \ot k$ to be irreducible over every field $k$ of characteristic $p$.  There is an easy-to-apply statement for $G = \SL_2$.  For $G = \SL_n$, Jantzen gives a necessary and sufficient condition,  but it is less easy to apply.  There are also sporadic results in one direction or another, such as consequences of the Linkage Principle like \cite[II.6.24]{Jan} or irreducibility when $\lambda$ is restricted and $\dim V(\lambda)$ is small.  Theorem \ref{thm:Levireduction} provides an easy way to get necessary conditions on $\lambda$ by taking various small $J$.  Writing $\lambda = \sum c_i \omega_i$ and taking $J = \{\alpha_i\}$ one can apply the $\SL_2$ criterion to constrain the possible values of $c_i$.  Taking $J$ to be pairs of adjacent roots of the same length allows one to reduce to the case of $A_2$, for which a lot is known, see \cite[II.8.20]{Jan}.
\end{rmk}

We mention the following related result that includes the case where $V(\la) \ot k$ is reducible.  

\begin{prop}
For every $\la \in X(T)_+$, every $J\subseteq \Delta$, and every field $k$, the irreducible representation $L_J(\la)$ of $L_J$ is a direct summand of $L(\la)\vert_{L_J}$.
\end{prop}

\begin{proof} For the sake of completeness we describe the analysis given in \cite[Section 8]{CN} which 
follows \cite{Smith} and \cite[II.5.21]{Jan}. There exists a weight space decomposition for the induced module 
given by 
$$
H^{0}(\lambda)=\left(\bigoplus_{\nu\in 
{\mathbb Z}J}H^{0}(\lambda)_{\lambda-\nu}\right) \oplus M.
$$ 
where $M$ is the direct sum of all weight spaces $H^{0}(\lambda)_{\sigma}$ where 
$\sigma\neq \lambda-\nu$ for any $\nu\in {\mathbb Z}J$. Furthermore, 
$H^{0}_{J}(\lambda)=\oplus_{\nu\in {\mathbb Z}J}H^{0}(\lambda)_{\lambda-\nu}$ with  
the aforementioned decomposition being $L_{J}$-stable. This allows us to identify an $L_{J}$-direct summand 
\begin{equation} 
H^{0}(\lambda)|_{L_{J}}\cong H^{0}_{J}(\lambda)\oplus M. 
\end{equation} 

By definition $L(\lambda)=\text{soc}_{G}(H^{0}(\lambda))$. This
implies that 
$\text{soc}_{L_{J}}L(\lambda)\subseteq \text{soc}_{L_{J}}(H^{0}(\lambda))$. 
Note that 
\begin{equation}
L_{J}(\lambda)=\text{soc}_{L_{J}}(H^{0}_{J}(\lambda))\subseteq 
\text{soc}_{L_{J}}(H^{0}(\lambda)).
\end{equation}
Now $L_{J}(\lambda)$ appears as an
$L_{J}$-composition factor of $L(\lambda)$ and 
$H^{0}(\lambda)$ with multiplicity one. 
Consequently, $L_{J}(\lambda)$ must occur 
in $\text{soc}_{L_{J}}L(\lambda)$. 

One can also apply the same argument for Weyl modules and see that 
\begin{equation} \label{eq;Weyldecomp}
V(\lambda)|_{L_{J}}\cong V_{J}(\lambda)\oplus M^{\prime}. 
\end{equation} 
for some $L_{J}$-module $M^{\prime}$. 
By an argument dual to the one in the preceding paragraph, we 
deduce that 
$L_{J}(\lambda)$ appears in the 
head of $L(\lambda)|_{L_{J}}$. The fact that 
$L_{J}(\lambda)$ has multiplicity one in $L(\lambda)$ now shows that $L_{J}(\lambda)$ 
is an $L_{J}$-direct summand of $L(\lambda)$. 
\end{proof}

\section{The case of fundamental weights}  \label{fund.sec}

We now verify Theorem \ref{MT} for every fundamental weight.  We abuse notation and write $V(\la)$ instead of $V(\la) \ot k$.

\subsection*{Type $A_{n}$ ($n \ge 1$)} In this case, all the fundamental weights are minuscule, so $V(\lambda)=L(\lambda)$ for all $\lambda=\omega_{j}$, $j=1,2,\dots,n$. 

\subsection*{Type $B_{n}$ ($n \ge 2$)} For $B_n$, we claim that \emph{$V(\omega_i)$ is reducible for $1 \le i < n$ and $\car k = 2$}.

The split adjoint group of type $B_n$ is $\SO(q)$ for a quadratic form $q$ on a vector space $X$ of dimension $2n+1$ where the tautological action on $X$ is $V(\omega_1)$, see \cite{KMRT} or \cite[\S23]{Borel}.  As $\car k = 2$, the bilinear form $b_q$ deduced from $q$ by the formula $b_q(x,y) := q(x+y) - q(x) - q(y)$ is necessarily degenerate with 1-dimensional radical, providing an $\SO(q)$-invariant line, call it $S$.

For $2 \le i < n$, we restrict to the Levi subgroup of type $B_{n-i+1}$ corresponding to $J = \{ \alpha_i, \alpha_{i+1}, \ldots, \alpha_n\}$.  By the previous paragraph, $V_J(\omega_i)$ is reducible in characteristic 2, hence $V(\omega_i)$ is by Theorem \ref{thm:Levireduction}.  

Alternatively, one can see the reducibility concretely by noticing that $V(\omega_i)$ has the same character and dimension as $\Lambda^i(V(\omega_1))$, because this is so in case $k = \C$.  In particular, $\Lambda^i(V(\omega_1))$ has a unique maximal weight, the highest weight of $V(\omega_i)$, and there is a nonzero $\SO(q)$-equivariant map $\phi \!: V(\omega_i) \to \Lambda^i(V(\omega_1))$.  As $S \wedge \Lambda^{i-1}(V(\omega_1))$ is a proper and $\SO(q)$-invariant subspace of $\Lambda^i (V(\omega_1))$, it follows that $V(\omega_i)$ is reducible.

\subsection*{Type $D_n$ ($n \ge 4$)} 
For type $D_n$, we claim that \emph{$V(\omega_i)$ is reducible for $2 \le i \le n - 2$ and $\car k = 2$.} 

The representation $V(\omega_2)$ has the same character and dimension as $\Lambda^2(V(\omega_1))$.  The alternating bilinear form $b_q$ deduced from the invariant quadratic form $q$ on $V(\omega_1)$ gives an invariant line in $\Lambda^2(V(\omega_1))$ --- i.e., $D_n$ maps into $C_n$, which is already reducible on $\Lambda^2(V(\omega_1))$ --- proving the claim for $i = 2$.

Alternatively,
$V(\omega_2)$ is the adjoint action on the Lie algebra of $\Spin_{2n}$ (when $\car k = 2$, this is distinct from $\Lie(\SO_{2n})$), and the center $S$ is a proper submodule, namely $\Lie(\mu_2 \times \mu_2)$ (if $n$ is even) or $\Lie(\mu_4)$ (if $n$ is odd).

For $2 < i \le n -2$, we may use either of the arguments employed in the $B_n$ case.

\subsection*{Type $C_{n}$ ($n \ge 3$)} For type $C_n$ with $n \ge 3$, 
 \cite[Th.~2(iv)]{PremetSup} gives that $V(\omega_2)$ is reducible when $\car k = p$ if and only if the prime $p$ divides $n$, compare \cite[p.~287]{Jan}.  For $\omega_i$ with $2 < i < n$, restricting to the Levi of type $C_{n-i+2}$ corresponding to $J = \{ \alpha_{i-1}, \alpha_i, \ldots, \alpha_n \}$ shows that $V_J(\omega_i)$ is reducible if $p$ divides $n-i+2$.  For $i = n$, we restrict to the Levi subgroup of type $C_2 = B_2$ corresponding to $J = \{ \alpha_{n-1}, \alpha_n \}$ to find that $V_J(\omega_n)$ is the 5-dimensional natural module for $B_2$, which is reducible in characteristic 2.  

\subsection*{Exceptional types} For exceptional types, tables of which fundamental weights $\omega$ have $V(\omega)$ reducible in which characteristics can be found in  \cite[p.~299]{Jan:first} or, for smaller dimensions, in \cite{Lubeck}.  These confirm our main theorem, and, in case $V(\omega) \ot k$ is reducible for some $k$, it is so for a $k$ with $\car k = 2$ or 3.

We remark that for the representations $V(\omega_i)$ of $E_8$ for $i \ne 8$, one can verify Theorem \ref{MT} by restricting to a Levi and using induction, instead of referring to \cite{Jan:first} or \cite{Lubeck} directly.

Because it is such an important example, we mention specifically that the adjoint representation $V(\omega_8)$ of $E_8$ is irreducible because $\Lie(E_8)$ is simple for every field, see \cite{St:aut} or \cite{Hogeweij}.  Here is an alternative argument provided to us by Gross: As $E_8$ is simply-laced, the Weyl group acts transitively on the roots, so the normalizer $N_{E_8}(T)$ of a split torus has an irreducible submodule in
the adjoint representation $\Lie(E_8)$ given by the sum of all the root spaces.
The miracle that is special to $E_8$ is that the Weyl group acts irreducibly on the submodule
$\Lie(T)$, which is the $E_8$-lattice mod $\car k$.\footnote{This is an illustration of a specific case, for $G$ the Weyl group of $E_8$, of Thompson's question mentioned in the introduction.}  Then the restriction of
the representation $\Lie(E_8)$ to $N_{E_8}(T)$ is the direct sum of two irreducible representations,
one of dimension 240 and the other of dimension 8.  Since $E_8$ has no nontrivial map into $\SL_8$, it  does not preserve either submodule, and so acts irreducibly on $\Lie(E_8)$.

This is in contrast to the case where $G$ is simple of type other than $E_8$, where $N_G(T)$ acts reducibly on $\Lie(T)$ for some characteristic (2 for types $B$, $C$, $D$, $E_7$ and $F_4$; 3 for $E_6$ and $G_2$; and dividing $n$ for type $A_{n-1}$).  And of course if $G$ has roots of different lengths and is simply connected, then for $\car k = 2$ or 3, the short roots generate a subalgebra of $\Lie(G)$ invariant under $G$, see e.g.~\cite[Lemma 3.2]{Hiss} or \cite[p.~1121]{St:aut}.

Here is yet another argument to see that $\Lie(E_8)$ is an irreducible representation for every field $k$.  Namely, it is a special case of the following observation: \emph{If $G$ is simple and simply connected over a field $k$, the center $Z$ of $G$ is \'etale\footnote{For example, this is true if $\car k$ is ``very good'' for $G$.}, and all the roots of $G$ have the same length, then $\Lie(G)$ is an irreducible representation of $G$.}  To prove this general statement, note that the natural map $\Lie(G) \to \Lie(G/Z)$ has kernel $\Lie(Z) = 0$, so is an isomorphism by dimension count.  But the domain is the Weyl module $V(\hr)$ and the codomain is its dual $V(\hr)^* = H^0(\hr)$ because of the assumption on the roots \cite[3.5]{G:vanish}.  Since $V(\hr) \cong H^0(\hr)$, they are irreducible $G$-modules.

%
%

\section{Type $B_n$, weight $\omega_1 + \omega_n$}  \label{B.sec}

Let $k$ be an algebraically closed field of characteristic $p \ge 0$.
Let $G= \Spin_{2n+1}(k)$ for $n \ge 2$.   The irreducible $G$-module $L(\omega_1)$ has dimension $2n+1$ if $\car k \ne 2$ and dimension $2n$ if $\car k = 2$. 
Moreover, the irreducible $G$-module $L(\omega_n)$ is the spin module for $G$ of dimension $2^n$.  In this section we show the following, which amounts to a specific case of Theorem \ref{MT}.   Although a different proof can be found in Lemmas 2.3.4 and 2.2.7 of \cite{BGT}, we include the proof below because it is a nice illustration of the use of finite group theory to prove a result about connected algebraic groups.

\begin{thm} \label{B.thm} Let $G=\Spin_{2n+1}(k)$ with $n \ge 2$. Then 
\[
\dim L(\omega_1 + \omega_n) = \begin{cases} 
2^n \cdot 2n   & \text{if $\car k$ does not divide $2n+1$;} \\
2^n \cdot (2n-1) & \text{if $\car k$ does divide $2n+1$.} 
\end{cases}
\]
\end{thm}  

The proof will appear at the end of the section. The analysis will entail the restriction of modules to a monomial subgroup of $\SO_{2n+1}$ via its lift to $\Spin_{2n+1}$ and 
the use of permutation modules for the alternating group.

Let  $U := L(\omega_1) \otimes L(\omega_n)$.   If $p=0$ (and so also for all but finitely
many $p$),  this is a direct sum of two composition factors $L(\omega_1 + \omega_n)$
and $L(\omega_n)$.   In particular, the Weyl module for the dominant
weight $\omega_1 + \omega_n$ has dimension $2n\cdot 2^n$.   

If $p = 2$ then as in \cite{St:rep}, $U$ is $L(\omega_1 + \omega_n)$, verifying the theorem.  We assume for the rest of the section that $p$ is odd.

Note that in $G/Z(G) = \SO_{2n+1}(k)$, there is a finite subgroup $X$ 
isomorphic to $A.A_{2n+1}$ where $A$ is an elementary
abelian $2$-group of rank $2n$ and $A_{2n+1}$ denotes the alternating group on $2n+1$ symbols.  The group $X$ is the derived subgroup of the group of orthogonal transformations
preserving an orthogonal set of $2n+1$ lines.   

Let $H$ denote the lift of $X$ to $G$.   Let $E$ be the lift of $A$ to $G$.   First we note:

\begin{lemma}  $E$ is extraspecial of order $2^{1+2n}$.
\end{lemma}

\begin{proof}   Since $X$ acts irreducibly on $E/Z(G)$,  $E$ is either elementary abelian
or extraspecial of the given order.  By induction it suffices to see that $E$ is nonabelian in the case $n=2$
(actually we could start with $n=1$).   This is clear since $\Spin_2(k)   \cong \Sp_4(k)$ and so contains
no rank $5$ elementary abelian $2$-groups.
\end{proof}

Note that $H/E \cong A_{2n+1}$.    Let $H_1$ be a subgroup of $H$ 
containing $E$ with $H_1/E \cong A_{2n}$.  

The group $E$ has a unique faithful irreducible module over $k$ of dimension $2^n$ that is the restriction of $L(\omega_n)$. (It is a tensor product of $n$ 2-dimensional representations of the central factors of $E$, cf. \cite[5.5.4, 5.5.5]{Gorenstein}.)
Since $Z(G)=Z(E)$ acts nontrivially on $U$, every composition factor for $E$ on $U$ is isomorphic
to $L(\omega_n)$.  It follows immediately that $L(\omega_1)$ and $L(\omega_n)$ are each irreducible for $H$.
Note also that $L(\omega_1)$ is induced from a linear character $\phi$ of $H_1$.  
Thus,  as an $H$-module,  $U \cong L(\omega_n) \otimes \phi_{H_1}^H$.   In fact, we see that we can replace $\phi$
by the trivial character of $H_1$:

\begin{lemma} \label{induced1} 
$U \cong L(\omega_n) \otimes k_{H_1}^H$ as an $H$-module.
\end{lemma}

\begin{proof}  It suffices to show that 
$L(\omega_n) \otimes \phi_{H_1} \cong  L(\omega_n) \otimes k_{H_1}$
as $H_1$-modules.   Note that they  are both irreducible  
since they are irreducible $E$-modules.     If $n=2$, the result is easy to
see (alternatively, one can modify the argument below).  So assume that $n > 2$. 

In fact, we observe that any $H_1$-module $V_1$ that is isomorphic to $L(\omega_n)$
as an $E$-module is isomorphic to $L(\omega_n)$ as an $H_1$-module.
This follows by noting that $\Hom_E(L(\omega_n), V_1)$ is $1$-dimensional and since
$H_1/E$ is perfect, $H_1/E$ acts trivially on this $1$-dimensional space, whence
 $\Hom_{H_1}(L(\omega_n), V_1)$ is also $1$-dimensional.  Since the two modules
 are irreducible, this shows they are isomorphic. 
\end{proof}

\begin{lemma}\label{lem.hom} $\dim \Hom_H(U,U)=2$.
\end{lemma}

\begin{proof}  This follows by Lemma \ref{induced1} and Frobenius reciprocity.
\end{proof}

Let $V$ be the unique nontrivial composition factor of $k_{A_{2n}}^{A_{2n+1}}$
(for $n > 1$).  This has dimension $2n$ if $p$ does not divide $2n+1$ and dimension $2n-1$
if $p$ does divide $2n+1$.  

By \cite{Dade} or \cite[Cor.~8.19]{Navarro}, we know:

\begin{lemma}  Viewing $V$ as an $H$-module (that is trivial on $E$), $L(\omega_n) \otimes_k V$ is 
irreducible. $\hfill\qed$
\end{lemma}

\begin{proof}[Proof of Theorem~\ref{B.thm}]
From the above, we see that $U$ has two $H$-composition factors if 
$p$ does not divide $2n+1$ and three composition factors
if $p$ does divide $2n+1$.   
This immediately implies that if $p$ does not divide $2n+1$, then $L(\omega_1 + \omega_n)$
is irreducible for $H$ and  has dimension $2n \cdot 2^n$ (whence also for $G$). 

Now assume that $p$ does divide $2n+1$. For sake of contradiction, suppose that $L(\omega_1 + \omega_n)$ has the same
dimension as the Weyl module $V(\omega_1 + \omega_n)$ for $G$, so $U$ has precisely  two nonisomorphic  composition factors as a
$G$-module, $L(\omega_1 + \omega_n)$ and $L(\omega_n)$.  Since $U$ is self-dual
it would be a direct sum of the two modules.

Recall $U$ has three $H$-composition factors (two isomorphic to $L(\omega_n))$.  Thus, 
the $G$-submodule $L(\omega_1 + \omega_n)$ must have two nonisomorphic
$H$-composition factors.  Again, since   $L(\omega_1 + \omega_n)$ is self dual,
this implies that $U$ is a direct sum of three simple $H$-modules.  This contradicts
 Lemma~\ref{lem.hom}  and completes the proof of Theorem \ref{B.thm}.
\end{proof}

Our analysis shows that when $p\mid 2n+1$ the Weyl module $V(\omega_{1}+\omega_{n})$ has two composition factors: $L(\omega_{1}+\omega_{n})$, $L(\omega_{n})$. 
Therefore, one can apply \cite[II.2.14]{Jan} to determine $\Ext^{1}$ between these simple modules. 

\begin{cor} \label{B.extcor} Let $G = \Spin_{2n+1}$. Then 
\[
\dim \Ext^{1}_{G}(L(\omega_1 + \omega_n),L(\omega_{n})) = \begin{cases} 
0 & \text{if $\car k$ does not divide $2n+1$;} \\
1 & \text{if $\car k$ does divide $2n+1$.} 
\end{cases}
\]
\end{cor}


%
%
%
%
%


\section{Proof of Theorem \ref{MT}} \label{proofofMain}

We now prove Theorem \ref{MT} by induction on the Lie rank of $G$. 

\subsection*{Type $A_1$} 
In case of rank 1, $G$ is $\SL_2$ and $V(d) \ot k$ is irreducible  if and only if, for $p = \car k$, $d+1 =cp^e$ for some $0 < c < p$ and $e \ge 0$ \cite[pp.~239, 240]{WinterSL2}.  (This can be seen by comparing dimensions: Write out $d$ in base $p$ as $d = \sum_i c_i p^i$.  Then $\dim V(d) = d+1$ whereas the irreducible module $L(d)$ over $k$ has dimension $\prod_i (c_i + 1)$ by Steinberg's tensor product theorem.)  As a consequence, for $d \ge 2$, it is impossible for $V(d) \ot \FF_p$ to be irreducible for both $p = 2$ and 3. 

An alternate argument (as noted by Andersen) can be provided if one does not require $p \le 3$. Choose a prime $p$ dividing $d$.  Now $\dim V(d)=d+1$ with $V(d)\twoheadrightarrow L(d)\cong L(d/p)^{(1)}$. But, $\dim L(d/p)^{(1)} \leq \frac{d}{p}+1 < d+1=\dim V(d)$. So $V(d)$ is reducible
in characteristic $p$. 

\subsection*{Reductions}
So suppose $\rank G \ge 2$ and Theorem \ref{MT} holds for all groups of lower rank. 

Write $\la = \sum c_i \omega_i$ with every $c_i \ge 0$.  If some $c_i > 1$, then taking $J = \{ \alpha_i \}$, the Levi subgroup $L_J$ has semisimple type $A_1$ and the restriction of $V_J(\la)$ to $L_J$ is reducible when $\car k$ is 2 or 3 by the 
argument for type $A_{1}$.  Therefore, by Theorem \ref{thm:Levireduction} we may assume that $c_i \in \{ 0, 1 \}$ for all $i$.

If $\la = 0$ or $\la = \omega_i$ for some $i$, then we are done by \S\ref{fund.sec}.  Hence, we may assume that at least two of the $c_i$'s are nonzero.

If there is a connected and proper subset $J$ of $\Delta$ such that $c_i \ne 0$ for at least two indexes $i$ with $\alpha_i \in J$, then we are done by induction and Theorem \ref{thm:Levireduction}.

\subsection*{Sums of extreme weights}
The remaining case is when the Dynkin diagram has no branches (i.e., $G$ has type $A$, $B$, $C$, $F_4$, or $G_2$) and $\la = \omega_1 + \omega_n$ is the sum of dominant weights corresponding to the simple roots at the two ends of the diagram.  For type $A_n$, $G = \SL_{n+1}$ and $V(\omega_1 + \omega_n)$  is the natural action on $\Lie(\SL_{n+1})$, the trace zero matrices.  If $p$ divides $n+1$, then the scalar matrices are a $G$-invariant subspace.
Type $B$ was handled in Theorem \ref{B.thm}.  

For type $C_n$ with $n \ge 3$, we restrict to the Levi subgroup of type $C_2$ and find that $V_J(\omega_1 + \omega_n)$ has dimension 5 and is reducible in characteristic 2.  Alternatively, as in \cite[\S11]{St:rep}, in characteristic 2 one finds $L(\omega_1 + \omega_n) \cong L(\omega_1) \ot L(\omega_n)$, which has dimension $n2^{n+1}$, whereas by the Weyl dimension formula, 
\[
\dim V(\omega_1 + \omega_n) = 7\cdot 2^{n-1} \cdot \frac{n(2n+1)}{n+3} \cdot \prod_{i=6}^{n+1} \frac{2i - 3}{i} \quad \text{for $n \ge 4$.}
\]

In the case of exceptional groups, for type $F_4$,  $V(\omega_1 + \omega_4)$ is reducible in characteristic 2 because it has dimension 1053, yet by \cite{St:rep} $L(\omega_1 + \omega_4) \cong L(\omega_1) \ot L(\omega_4)$ has dimension $26^2 = 676$.  For type $G_2$, $\dim V(\omega_1 + \omega_2) = 64$ yet by Steinberg in characteristic 3 $L(\omega_1 + \omega_2)$ has dimension $7^2 = 49$.  Alternatively, one can refer to \cite[Tables A.49, A.50]{Lubeck}.
This completes the proof of Theorem \ref{MT}.\hfill\qed


\section{Complements to Theorem \ref{MT}}

\subsection*{Invariant bilinear forms} Let $G$ be a split reductive group over $\ZZ$ with a split maximal torus $T$ and $\la \in X(T)_+$.  A $G$-invariant bilinear form $b$ on the Weyl module $V(\la)$ corresponds to a $G$-equivariant homomorphism $\delta_b \!: V(\la) \to V(\la)^*$ via $\delta_b(v)(v') = b(v,v')$.  The map $\delta_b$ is determined by what it does to a highest weight vector in $V(\la)$ and in order that $\delta_b$ be nonzero it must be that $\la$ is a weight of $V(\la)^*$, and in particular that $\la \le -w_0 \la$, from which it follows that $\la = -w_0 \la$.  As the highest weight spaces in $V(\la)$ and $V(\la)^*$ are rank 1 $\ZZ$-modules, we conclude that the space of $G$-invariant bilinear forms on $V(\la)$ is $\ZZ$ if $\la = -w_0 \la$ and 0 otherwise.

So suppose $\la = -w_0 \la$ and let $b$ be an indivisible $G$-invariant bilinear form on $V(\la)$ --- it is determined up to sign.  For each field $k$, the map $\delta_b \!: V(\la) \ot k \to V(\la)^* \ot k$ has kernel the unique maximal proper submodule of $V(\la) \ot k$, see \cite[II.2.4, II.2.14]{Jan}, so \emph{$b \ot k$ is nondegenerate if and only if $V(\la) \ot k$ is irreducible.}  Therefore, Theorem \ref{MT} gives:
\begin{cor} \label{Killing.cor}
Suppose, in the notation of the previous two paragraphs, that  $G$ is simple and split and $\la = -w_0 \la$.  Then $b \ot k$ is nondegenerate for every field $k$ if and only if
\begin{enumerate}
	\renewcommand{\theenumi}{\alph{enumi}}
\item $\la$ is minuscule; or 
\item $G = E_8$ and $\la$ is the highest root $\hr$.
\end{enumerate}
\end{cor}

\subsection*{Failure of converse to Theorem \ref{thm:Levireduction}} As another complement to Theorem \ref{MT}, we make precise the settings where the converse to ``$V(\la)$ irreducible implies $V_J(\la)$ irreducible'' fails.

\begin{thm}\label{qm}
 Let $\lambda$ be a dominant weight for $G$ a split, simple, and simply connected group over $\ZZ$ with $\rank G > 1$.  Then $V_J(\la) \otimes k$ is irreducible for every $J \subsetneq \Delta$ if and only if one of the following occurs:
 \begin{enumerate}
 \renewcommand{\theenumi}{\alph{enumi}}
 \item \label{qm.min} $\la$ is minuscule.
 \item \label{qm.hsr} $\la$ is the highest short root $\hsr$.
 \item \label{qm.B} $\Phi=B_n$ and $\lambda = \omega_1 + \omega_n$.
 \item \label{qm.G} $\Phi=G_2$ and $\lambda = \omega_2$ or $\omega_1 + \omega_2$.
 \end{enumerate}
\end{thm}

\begin{proof} First assume that one of conditions \eqref{qm.min}--\eqref{qm.G} holds. Then one can directly verify that $V_J(\la) \otimes k$ is irreducible for every $J \subsetneq \Delta$ 
(by Theorem \ref{MT}).

On the other hand, suppose that $V_J(\la) \otimes k$ is irreducible for every $J \subsetneq \Delta$.  Write $\la = \sum_i a_i \omega_i$.  If some $a_i > 1$ or at least three of
the $a_i$ are nonzero, then by Theorem \ref{MT}, we see that $V_J(\la)$ is not irreducible with $J$ obtained by removing an end node other than $i$
in the first case or any end node in the second case.

Next consider the case when $\la = \omega_i + \omega_j$, $i\neq j$. The result follows unless $\{i,j \}$ correspond to all the end nodes.   If there are three end nodes, this is not possible.
Thus, we only need consider types $A, B, C, F$ and $G$.  If $\Phi=A_n$, this leads to \eqref{qm.hsr}.  If $\Phi=B_n$, this leads to \eqref{qm.B}. Moreover, if $\Phi=C_n, n \ge 3$, then $V_J(\la) \ot k$ is reducible for $J= \Delta-\{\alpha_{1} \}$ and $\car k = 2$.
Similarly, if $\Phi=F_4$, $V_J(\la) \ot k$ is reducible for $J=\Delta-\{ \alpha_{4} \}$ and $\car k = 3$.  If $\Phi= G_2$, this leads to one of the cases in \eqref{qm.G}.  

It remains to consider the case that $\la = \omega_i$ for some $i$.  If $\Phi=A_n$, then $\omega_i$ is minuscule.   If $G$ has rank $2$, then
removing a single node gives a Levi of type $A_{1}$ and so we have irreducibility as in \eqref{qm.min}, \eqref{qm.hsr}, and \eqref{qm.G}.     So assume that $\Phi$ is not of type $A_n$ and has rank at least $3$.   It suffices to check that for any $J$
obtained by removing an end node that $V_J(\la)$ irreducible implies that $\la$ is either minuscule or $\la = \hsr$.  

Suppose that $\Phi=D_n$, $n \ge 4$.  If $\la$ is not minuscule and $\la \ne \hsr
= \omega_2$, then we can remove the first node and see that $V_J(\la) \ot k$ is reducible for $\car k = 2$.

It remains to consider types $B,C, E$ and $F$. If $i$ does not correspond to an end
node, then we can choose $J$ in such a way that the Levi factor $L_J$ of the reduced
system does not have type $A_n$ and $V_J(\la)$ does not correspond to an end node,
whence by Theorem \ref{MT},  $V_J(\la)$ is not irreducible.    

If $G$ has type $B_n$ or $C_n$, then $\omega_1$ and $\omega_n$ either correspond to
the short root or are minuscule for $L_J$. In the case when $G$ has type $E_6$,
then $\omega_i$ corresponding to an end node is either $\hsr$ or minuscule for $L_J$.  
If $G$ has type $F_4$ or $E_n, n \ge 7$, then one checks the only end node satisfying
the hypotheses is $\hsr$.   
\end{proof}

\subsection*{Connection with $B$-cohomology} Let $B$ be the Borel of $G$ corresponding to the negative roots. For 
$2\rho$ the sum of the positive roots and $N$ the number of positive roots, one can use Serre duality to show that 
\begin{equation}
\text{Hom}_{G}(k,V(-w_{0}\lambda))\cong \text{Ext}^{N}_{B}(k,\lambda-2\rho)\cong \text{H}^{N}(B,\lambda-2\rho),
\end{equation} 
see \cite{HN} and \cite[Theorem 5.5]{GN}.  

For $\la = \hr$ the highest root, $\hr = -w_0 \hr$ and $V(\hr)$ is the Lie algebra of the simply connected cover of $G$. 
The adjoint representation for $E_{8}$ is simple for all $p>0$, so 
$$\text{H}^{120}(B,\widetilde{\alpha}-2\rho)=0.$$ 
On the other hand, 
if $G$ is of type $A_{n}$ then 
$$ 
\text{H}^{n(n+1)/2}(B,\widetilde{\alpha}-2\rho)\cong \begin{cases} k & \text{if $p\mid n+1$} \\
                                                           0  & \text{if $p\nmid n+1$.} 
                                                           
\end{cases} 
$$
Similar statements can be formulated for other types. 

The calculation of the $B$-cohomology with coefficients in a one-dimensional representation is an open problem in general. Complete answers are known for 
degrees $0$, $1$, and $2$ and for most primes in degree $3$. See \cite{BNP} for a survey.
                                                            
\subsection*{Quantum Groups} For quantum groups (Lusztig ${\mathcal A}$-form) at roots of unity, 
one can ask when the quantum Weyl modules are globally irreducible. The Weyl modules 
with minuscule highest weights will yield globally irreducible representations.Ê
One can prove an analog of Theorem~\ref{thm:Levireduction} to use Levi factors to reduce to 
considering fundamental weights
or weights of the form $\omega_{1}+\omega_{n}$. 

For type $A_n$, if the root of unity has order 
$l$ and $l\mid  n+1$ then
$V(\omega_{1}+\omega_{n})$ is not simple (see \cite{Fayers}). 
This uses representation theory of the
Hecke algebra of type $A_n$.
From this one can prove the analog of our main theorem 
(Theorem~\ref{MT}) for quantum
groups in the $A_n$ case.

In order to handle root systems other than $A_{n}$, more detailed information needs to be worked out such as the 
the tables given in \cite{Jan:first} and analogs of results for Weyl modules in type $C_{n}$ as given in 
\cite{PremetSup}.

\subsection*{Further Directions} 
Suppose now that $G$ is a split simply connected algebraic group over $\ZZ$ and $\lambda$ is a dominant weight.  In a preliminary version of this manuscript, we asked to what extent is the following statement true: \emph{If $\mu$ is a dominant weight that is maximal among the dominant weights $< \lambda$, then there is a field $k$ such that $V(\lambda) \otimes k$ has $L(\mu)$ as a composition factor.}  Certainly, it is false for $G = E_8$, $\lambda$ the highest root, and $\mu = 0$.  Jantzen has recently shown in \cite{Jan:max} that, apart from this one counterexample, the statement holds when $G$ is simple.  Note that, in contrast to Theorem \ref{MT}, this result does not include an upper bound on $\car k$ that only depends on the rank of $G$.  For example, take $G = \SL_2$ and pick a prime $p$ and a $d > p$ not divisible by $p$.  Then $d-2$ is a weight of the irreducible representation $L(d)$ with highest weight $d$ over $\FF_p$, so $L(d-2)$ is not in the composition series for $V(d) \ot \FF_p$.


\section{Appendix: the Killing form over $\ZZ$}

\subsection*{The reduced Killing form} 
Let $G$ be a split simple algebraic group over $\ZZ$.  There is a canonical indivisible $G$-invariant bilinear form on $\Lie(G)$, the \emph{reduced Killing form}, which we denote $b_G$.  In this appendix, we prove the following result, which is similar in flavor to Theorem \ref{MT} and answers a question raised by George Lusztig.

\begin{thm} \label{Killing}
The reduced Killing form on $\Lie(G) \ot k$ is nondegenerate for every field $k$ if and only if $G$ is one of the following groups:
\begin{enumerate}
\renewcommand{\theenumi}{\alph{enumi}}
\item $E_8$;
\item \label{Killing.SO} $\SO_{2n}$ for some $n \ge 4$; 
\item \label{Killing.HSpin} $\HSpin_{2n}$ for $n$ divisible by $4$; or
\item \label{Killing.SL} $\SL_{m^2}/\mu_m$ for some $m > 1$.
\end{enumerate}
\end{thm}

We actually prove something more.  The Lie algebra of $G$ is a free $\ZZ$-module, so it has a basis $v_1, \ldots, v_n$ for some $n$.  The determinant of $b_G$ (denoted $\det b_G$) is the determinant of the matrix with $(i,j)$-entry $b_G(v_i, v_j)$; note that $\det b_G$, as an element of $\ZZ$, does not depend on the choice of basis.  Furthermore, $b_G \ot k$ is degenerate if and only if $\det b_G = 0$ in $k$.  The point of this appendix is to calculate $\det b_G$ for every simple $G$ over $\ZZ$, from which Theorem \ref{Killing} quickly follows.

\subsection*{The case where $G$ is simply connected} In case $G$ is simply connected, the Killing form on $\Lie(G)$ --- a bilinear form over $\ZZ$ --- is divisible by $2\hvee$ for $\hvee$ the dual Coxeter number of $G$; we define $b_G$ to be the quotient.  It is even and indivisible \cite[Prop.~4]{GrossNebe}.  It is natural to call $b_G$ the reduced Killing form because it is obtained from the Killing form by dividing by the greatest common divisor of its values.  (Note that $b_G$ has the advantage that $b_G \ot k$ is nonzero for every $k$, a property not satisfied by the Killing form.)

For simply connected $G$, $\det b_G$ was calculated in \cite[I.4.8(a)]{SpSt}.  Specifically, let $N$ denote the number of positive roots, $N_s$ denote the number of short roots, and $N_{ss}$ the number of short simple roots.  Put $c$ for the ratio of the square-lengths of the long to short roots (so $c \in \{ 1, 2, 3 \}$) and $f$ for the determinant of the Cartan matrix.  Let $T$ be a split maximal torus in $G$ over $\ZZ$.
Then $\Lie(G)$ is an orthogonal sum of $\ft := \Lie(T)$ and a subspace $\fn$ spanned by the root subalgebras of $G$ with respect to $T$.  One checks that $\det b_G\vert_{\fn} = (-1)^N c^{N_{ss}}$.  The Lie algebra $\ft$ is naturally identified with the coroot lattice $Q^\vee$, and this identifies the restriction of $b_G$ to $\ft$ with the Weyl-group invariant bilinear form such that $b_G(\alpha^\vee, \alpha^\vee) = 2$ for every short coroot $\alpha^\vee$.  In summary, one finds that
\begin{equation} \label{SpSt.formula}
\det b_G = (-1)^N c^{N_s + N_{ss}} f \quad \text{for $G$ simply connected.}
\end{equation}
The value of $\det b_G$ can be found in Table \ref{redkill.table}.  The conflicting values given in the table on p.~634 of \cite{GrossNebe} are typos.
\begin{table}[hbt]
\begin{tabular}{c|c|ccc}
Killing-Cartan&$G$ simply connected&\multicolumn{2}{c}{$G$ adjoint} \\
type of $G$&$\det b_G$&$e$&$\det b_G$ \\ \hline
$A_n$&$(-1)^{(n^2+n)/2} (n+1)$&$n+1$&$(-1)^{(n^2+n)/2}(n+1)^{n^2+2n-1}$ \\ 
$B_n$&$(-1)^n 2^{2n+2}$&1&$(-1)^n 2^{2n}$ \\ 
$C_n$&$(-1)^n 2^{2n^2 - n}$&$\rho(n)$&$(-1)^n 2^{2n^2-n-2} \rho(n)^{2n^2 + n}$ \\ 
$D_n$ ($n \ge 4$)&$2^2$&$2\rho(n)$&$2^{2n^2-n-2} \rho(n)^{2n^2 - n}$ \\ 
$G_2$&$3^7$&\multicolumn{2}{c}{$\cdots$} \\ 
$F_4$&$2^{26}$&\multicolumn{2}{c}{$\cdots$} \\ 
$E_6$&3&3&$3^{77}$ \\ 
$E_7$&$-2$&2&$-2^{132}$ \\ 
$E_8$&1&\multicolumn{2}{c}{$\cdots$}
\end{tabular}
\caption{Determinant of the reduced Killing form for $G$ simply connected or adjoint.  The value of $e$ is taken from \cite[\S3]{G:vanish}.  The notation $\rho(n)$ means 1 if $n$ is even and 2 if $n$ is odd.} \label{redkill.table}
\end{table}

\subsection*{Definition of reduced Killing form for $G$ not simply connected}  Suppose that $G$ is split simple over $\ZZ$ and let $f \!: \Gt \to G$ denote the simply connected cover in the sense of \cite[XXI.6.2.6, XXII.4.3.3]{SGA3:new}.
Note that $\Gt$ and $G$ may have distinct Lie algebras;
the natural map $\Lie(\Gt) \ot k \to \Lie(G) \ot k$ has kernel $\Lie(\ker f \times k)$, which may be nonzero.
The differential $\df \!: \Lie(\Gt) \to \Lie(G)$ gives an isomorphism $\df_\QQ \!: \Lie(\Gt) \ot \QQ \to \Lie(G) \ot \QQ$; pushing forward $b_\Gt$ gives a $G$-invariant bilinear form $\hat{b}_G$ on $\Lie(G) \ot \QQ$ and we define the reduced Killing form on $\Lie(G)$ to be $b_G := e \cdot \hat{b}_G$, where $e$ is the smallest positive rational number such that the resulting $b_G$ has integer values.  It follows from the indivisibility of $b_\Gt$ that $e$ is an integer.

Pick a pinning of $G$ with respect to a split maximal torus $T$, and fix a corresponding pinning of $\Gt$ with respect to the maximal torus $\Tt := f^{-1}(T)^\circ$.  The two groups have the same root system and $\df$ restricts to give an isomorphism for each of the 1-dimensional root subalgebras of $\Lie(G)$ with the corresponding root subalgebra of $\Lie(\Gt)$.  The map $\df$ embeds $\Lie(\Tt) = Q^\vee$ in $\Lie(T)$.

In case $G$ is adjoint, \cite[I.4.8(b)]{SpSt} says that $\det \hat{b}_G = (\det b_G)/f^2$ and therefore $\det b_G = e^{\dim G} (\det b_\Gt)/f^2$.  This gives the values in the last column of Table \ref{redkill.table}.  Note that $\Lie(T)$ is naturally identified with the lattice $P^\vee$ of weights for the dual root system.

\subsection*{Groups that are neither simply connected nor adjoint}  It remains to treat the case where $G$ is neither simply connected nor adjoint.  Recall that for even $n > 4$, the simply connected group $\Spin_{2n}$ has two non-isomorphic quotients by a central $\mu_2$:  $\SO_{2n}$ and one more called a \emph{half-spin group}; we denote it by $\HSpin_{2n}$.  (For $n = 4$, $\HSpin_8$ is defined, but it is isomorphic to $\SO_8$.)  
So 
suppose $G$ is $\SL_n / \mu_m$ for some $m \mid n$, $\SO_{2n}$ for some $n \ge 4$, or $\HSpin_{2n}$ for some even $n \ge 4$.
For each of these three possibilities, we determine $\hat{b}_G\vert_{\Lie(T)}$ and $e$.

As all roots have the same length, we use the canonical identification of the root system with its dual and calculate $\Lie(T)$ as a sublattice of the weight lattice $P$.  We may identify $\Lie(T)$ by noting that for each possibility for $G$, $\Lie(T)/Q$ is cyclic, so it suffices to find a fundamental weight $\omega$ such that $\Lie(T)$ is generated by $\omega$ and $Q$, and to find a simple root $\alpha$ and $c \in \NN$ so that $c\omega$ is in $Q$ and is a sum of $\alpha$ and a linear combination of the other simple roots.  Then $\{ \omega \} \cup \Delta \setminus \{ \alpha \}$ is a basis for $\Lie(T)$.  We take, with roots numbered as in Table \ref{dynks.table}:
\begin{itemize}
\item for $G = \SL_n / \mu_m$: $\omega = (n/m)\omega_{n-1}$, $\alpha = \alpha_1$, $c = m$;
\item for $G = \SO_{2n}$: $\omega = \omega_1$, $\alpha = \alpha_n$, $c = 2$; or
\item for $G = \HSpin_{2n}$: $\omega = \omega_n$, $\alpha = \alpha_1$, $c = 2$,
\end{itemize}
see \cite[3.6, 3.7, 5.2]{G:vanish}, although a slightly different choice of basis was taken for $\HSpin_{2n}$ there.  (We can now see $e$: as $\hat{b}_G(\omega, \omega) = n(n-1)/m^2$ (as in \cite[Lemma 5.2]{G:vanish}), $1$, and $n/4$ respectively, and $\hat{b}_G(\omega, Q) \subseteq \ZZ$, we find that $e = m/\gcd(m, n/m)$, $1$, and $\rho(n/2)$ respectively.) 

We can write down the Gram matrix for $\det \hat{b}_G\vert_{\Lie(T)}$ with respect to the basis ordered by taking the simple roots in the Bourbaki ordering, deleting $\alpha$, and appending $\omega$.  This leads to $\det b_G$ via the formula
\begin{equation}\label{e.det}
\det b_G = (-1)^N e^{\dim G} \det \hat{b}_G\vert_{\Lie(T)},
\end{equation}
which follows from the definition of $e$ and the fact that $\df$ restricts to an isomorphism on the span of the root subalgebras. 

For $G = \SL_n / \mu_m$, the Gram matrix is $(n-1)$-by-$(n-1)$, with a Cartan matrix of type $A_{n-2}$ in the upper left corner and the right column and bottom row are both $(0, \ldots, 0, n/m, n(n-1)/m^2)$, so its determinant is $n/m^2$.  Equation \eqref{e.det} gives that 
\begin{equation} \label{sl.det}
\det b_{\SL_n/\mu_m} = (-1)^{\binom{n}{2}} \frac{n}{m^2} \left( \frac{m}{\gcd(m, n/m)} \right)^{n^2 - 1}.
\end{equation}

For $G = \SO_{2n}$, the Gram matrix is $n$-by-$n$, has a Cartan matrix of type $A_{n-1}$ in the upper left corner, and has right column and bottom row both $(1, 0, \ldots, 0, 1)$.  So its determinant is 1, and $\det b_G = 1$.

For $G = \HSpin_{2n}$, the Gram matrix is $n$-by-$n$, has a Cartan matrix of type $D_{n-1}$ in the upper left corner, and has right column and bottom row both $(0, \ldots, 0, 1, n/4)$.  So its determinant is 1, and $\det b_G = \rho(n/2)$.

\begin{proof}[Proof of Theorem \ref{Killing}] Given the calculation of $\det b_G$ above, it suffices to consider the case $G = \SL_n / \mu_m$ where $n \ge 4$.  If $n = m^2$, then $\det b_G = \pm 1$ and $b_G$ is nondegenerate.  Conversely, assume $b_G$ is nondegenerate, so the restriction to the root subalgebras must be nondegenerate, i.e., $e = 1$, equivalently, $\gcd(m, n/m) = m$, equivalently, $m^2$ divides $n$.  Then $\det b_G = \pm n/m^2$.
\end{proof}

\begin{rmk*}
The subdivision of groups of type $D_n$ for $n$ even into the cases $n \equiv 0, 2 \bmod 4$, as seen in Theorem \ref{Killing}\eqref{Killing.HSpin}, can also be seen in the representation theory of these groups: the half-spin representation of $\Spin_{2n}$ over $\C$ is orthogonal for $n$ divisible by 4 and symplectic for $n \equiv 2 \bmod{4}$, see \cite[Ch.~VIII, Table 1]{Bou:g7} or \cite[8.4]{KMRT}.
\end{rmk*}

\begin{eg}
Let $G := \SL_{p^d}/\mu_p$ for some prime $p$ and some $d \ge 3$.  Then the reduced Killing form $b_G \ot k$ on $\Lie(G) \ot k$ is degenerate when $\car k = p$, where the element $\omega$ in $\Lie(T) \ot k$ is in the radical.  Yet $\Lie(G) \ot k$ is self-dual because $\Lie(G) \ot k$ is a direct sum of the self-dual representations $L(\hr)$ and $k$ by \cite[9.4]{Hum:p}.  Furthermore, by \eqref{sl.det}, $b_G \ot k'$ is non-degenerate for every field $k'$ of characteristic different from $p$.
\end{eg}

\bibliographystyle{amsalpha}
\bibliography{e8-minuscule}

%
%
%
%
%
%
%
%
%
%

\end{document}